\documentclass[bibtex,12pt,en]{elegantpaper}
\usepackage{extarrows}
\usepackage{esint}
\usepackage{mathrsfs}
\numberwithin{equation}{section}
\allowdisplaybreaks[4]
\everymath{\displaystyle}

\title{Spectral Improvements of Geometric Inequalities on Closed K\"ahler Manifolds}

\author{Sayantan Chakraborty \and Xiaodong Wang \and Tian Wu}

\begin{document}

\date{}
\maketitle

\renewcommand{\thefootnote}{\fnsymbol{footnote}}


\begin{abstract}
    \indent Let $(M,g,J)$ be a closed K\"ahler manifold satisfying $\operatorname{Ric}\geqslant g$. We establish improved Liouville theorems for the Euler--Lagrange equations associated with the Beckner--Sobolev inequalities by incorporating the first positive eigenvalue of the $\bar\partial$-Laplacian into a differential-identity argument. As a consequence, we obtain improved Sobolev and Beckner inequalities that refine the known Riemannian and K\"ahler estimates when the first eigenvalue is sufficiently large. We also derive new upper bounds for the diameter of $(M,g)$.
    \keywords{the Beckner--Sobolev inequality, K\"ahler manifold, Liouville theorems, the differential identity method.}\\
    \textbf{2020 Mathematics Subject Classification:} 53C55, 53C21, 35B53.
\end{abstract}

\section{Introduction}

Let $(M,g)$ be a closed Riemannian manifold of dimension $n\geqslant 2$. A fundamental inequality in geometric analysis is the Beckner--Sobolev inequality (see, for instance, Hebey \cite{Heb99}):
$$\frac{1}{q-2}\Big[\Big(\fint_M|f|^q\Big)^{\frac 2 q}-\fint_M|f|^2\Big]\leqslant C\fint_M|\nabla f|^2,$$
where $q\in\begin{cases}
    [1,2)\cup(2,\frac{2n}{n-2}], & n\geqslant3,\\
    [1,2)\cup(2,+\infty), & n=2,
\end{cases}$ and $f\in H^1(M)$.

When $\mathrm{Ric}\geqslant(n-1)\kappa g$ for $\kappa>0$, Bidaut-V\'eron and V\'eron \cite{BV91} proved a Liouville theorem for the Euler--Lagrange equation associated with the Beckner--Sobolev inequality by establishing a differential identity. Their Liouville theorem yields the inequality with $C=\frac{1}{n\kappa}$. This constant is sharp in the class under consideration, since equality is attained on the standard sphere $\mathbb S^n(\frac{1}{\sqrt\kappa})$; see Beckner \cite{Bec93}.

The constant $C$ can be further improved by incorporating additional geometric information, such as the diameter or the first Laplace eigenvalue. Using a rearrangement argument, Ilias \cite{Ili83} reduced the problem on $M$ to one on the model sphere $\mathbb S^n(\frac{1}{\sqrt\kappa})$, where the inequality was already known; see \cite{BV91,Bec93}. The key ingredient in Ilias' proof is the L\'evy--Gromov isoperimetric inequality. B\'erard, Besson, and Gallot \cite{BBG85} improved the isoperimetric inequality and obtained a better constant $C=\frac{1}{n\kappa\tau ^{2}}$, where $\tau=\Big(\frac{\int_0^{\frac\pi 2}\cos^{n-1}t\,\mathrm dt}{\int_0^{\frac{\sqrt\kappa D}{2}}\cos^{n-1}t\,\mathrm dt}\Big)^{\frac 1 n}$  and $D=\operatorname{diam}(M,g)$. The maximal diameter theorem implies that $\tau\geqslant1$. Equality holds if and only if $(M,g)$ is isometric to $\mathbb S^n(\frac{1}{\sqrt\kappa})$. 

A second approach uses the first eigenvalue $\Lambda$ of $-\Delta$. Licois and V\'eron \cite{LV95}, Fontenas \cite{Fon97}, Dolbeault, Esteban, and Loss \cite{DEL14} improved Liouville theorems by using the Poincar\'e inequality. These results yield the improved constant
$$C=\frac{1}{\theta n\kappa+(1-\theta)\Lambda},\quad\theta=\frac{(n-1)^2(q-1)}{q+n^2+2n-1}\in[0,1].$$

\begin{theorem}[\cite{LV95,Fon97,DEL14}]\label{thm:Riemann-Sobolev}
    Let $(M,g)$ be a closed $n$-dimensional Riemannian manifold satisfying $\operatorname{Ric}\geqslant(n-1)g$, and let $\Lambda$ denote the first positive eigenvalue of $-\Delta$. Assume that
    $$\alpha\in\begin{cases}
        (1,\frac{n+2}{n-2}],& n\geqslant3,\\
        (1,+\infty),& n=2,
    \end{cases}\quad\text{and }0<\lambda<\frac{\theta n+(1-\theta)\Lambda}{\alpha-1}\text{ with }\theta=\frac{(n-1)^2\alpha}{\alpha+n^2+2n}.$$
    Then every positive solution of $-\Delta u+\lambda u=u^\alpha$ is constant; more precisely, $u\equiv\lambda^{\frac{1}{\alpha-1}}$.
\end{theorem}

\begin{theorem}[\cite{DEL14}]\label{thm:Riemann-Beckner}
    Let $(M,g)$ be a closed $n$-dimensional Riemannian manifold satisfying $\operatorname{Ric}\geqslant(n-1)g$, and let $\Lambda$ denote the first positive eigenvalue of $-\Delta$. Assume that
    $$0<\alpha<1\quad\text{and }0<\lambda<\frac{\theta n+(1-\theta)\Lambda}{1-\alpha}\text{ with }\theta=\frac{(n-1)^2\alpha}{\alpha+n^2+2n}.$$
    Then every positive solution of $\Delta u+\lambda u=u^\alpha$ is constant; more precisely, $u\equiv\lambda^{\frac{1}{\alpha-1}}$.
\end{theorem}

In the variational formulation, the exponents are related by $q=\alpha+1$. The two theorems above show that a sufficiently large first eigenvalue can yield a substantial improvement in the Sobolev constant. It is therefore natural to ask whether sharper estimates can be obtained for special classes of manifolds with positive Ricci curvature. The K\"ahler case is especially interesting. Motivated by these observations, we investigate the K\"ahler case by incorporating spectral information into the differential-identity method.

Throughout the remainder of the paper, let $(M,g,J)$ be a closed K\"ahler manifold of complex dimension $m\geqslant2$ satisfying $\operatorname{Ric}\geqslant g$. For any sufficiently smooth function $f$, we adopt the following notation:
\begin{itemize}
    \item $\square$ denotes the $\bar\partial$-Laplacian, i.e. $\square f:=\nabla^i\nabla_i f=\frac 1 2\Delta f$, where $\Delta$ is the real Laplacian.
    \item $|\partial f|^2$ denotes $\nabla_i f\nabla^i f$, which equals $\frac12|\nabla f|^2$.
\end{itemize}
In local coordinates, we denote the metric tensor by $g_{i\overline j}$ and the Ricci tensor by $R_{i\overline j}$. Indices are raised and lowered using $g_{i\overline j}$ and $g^{i\overline j}$, respectively, and repeated indices are summed over. We write $\fint_M h:=\frac{1}{\operatorname{Vol}(M,g)}\int_M h$ for the average integral over $M$.

For convenience, throughout the remainder of the paper we use $|\partial f|^2$ and $\square f$ in place of $|\nabla f|^2$ and $\Delta f$, respectively. We rewrite the Beckner--Sobolev inequality as
\begin{equation}\label{BS}
    \frac{1}{q-2}\Big[\Big(\fint_M|f|^q\Big)^{\frac 2 q}-\fint_M|f|^2\Big]\leqslant A\fint_M|\partial f|^2,
\end{equation}
where $q\in[1,2)\cup(2,\frac{2m}{m-1}]$. The Euler--Lagrange equations associated with \eqref{BS} reduce the determination of the sharp constant $A$ to Liouville-type classifications of positive solutions to the following equations:
\begin{equation}\label{eq:Sobolev}
    -\square u+\lambda u=u^\alpha,\quad1<\alpha\leqslant\frac{m+1}{m-1},\quad\lambda>0,
\end{equation}
\begin{equation}\label{eq:Beckner}
    \square u+\lambda u=u^\alpha,\quad0<\alpha<1,\quad\lambda>0.
\end{equation}
For a smooth function $u$, we write $u_i:=\nabla_i u$, $u_{\bar i}:=\nabla_{\bar i}u$, and use analogous notation for higher covariant derivatives. Indices are raised using the metric.

Let $\Lambda$ be the first positive eigenvalue of $-\square$. We shall use the following Bochner identity:
$$\nabla^i(\nabla_j\nabla_i f\nabla^j f)=|\partial\partial f|^2+\operatorname{Ric}(\partial f,\bar\partial f)+\langle\partial\square f,\partial f\rangle.$$
Let $f$ be an eigenfunction of $-\square$ with eigenvalue $\Lambda$, i.e. $-\square f=\Lambda f$. Integrating the Bochner identity and using $-\square f=\Lambda f$, we obtain
$$\Lambda\int_M|\partial f|^2\geqslant\int_M\operatorname{Ric}(\partial f,\bar\partial f)\geqslant\int_M|\partial f|^2.$$
Since $f$ is nonconstant, $\int_M|\partial f|^2>0$, and hence $\Lambda\geqslant1$, a fundamental result due to Aubin \cite{Aub78}. The lower bound $\Lambda\geqslant 1$, together with the necessary condition obtained by linearizing \eqref{BS} around constant functions, motivates the following two conjectures.
\begin{conjecture}\label{conj:Sobolev}
    Suppose $(M,g,J)$ is a closed K\"ahler manifold of complex dimension $m\geqslant2$, $\operatorname{Ric}\geqslant g$, $1<\alpha\leqslant\frac{m+1}{m-1}$, and $0<\lambda<\frac{1}{\alpha-1}$. Then every positive solution of \eqref{eq:Sobolev} is constant; more precisely, $u\equiv\lambda^{\frac{1}{\alpha-1}}$.
\end{conjecture}

\begin{conjecture}\label{conj:Beckner}
    Suppose $(M,g,J)$ is a closed K\"ahler manifold of complex dimension $m\geqslant2$, $\operatorname{Ric}\geqslant g$, $0<\alpha<1$, and $0<\lambda<\frac{1}{1-\alpha}$. Then every positive solution of \eqref{eq:Beckner} is constant; more precisely, $u\equiv\lambda^{\frac{1}{\alpha-1}}$.
\end{conjecture}

We omit the case $m=1$, because both statements follow from the corresponding Riemannian results in dimension $n=2$. In view of the Euler--Lagrange equations associated with \eqref{BS}, Conjectures \ref{conj:Sobolev} and \ref{conj:Beckner} would imply the following sharp Beckner--Sobolev inequality.

\begin{conjecture}\label{conj:BS}
    Suppose $(M,g,J)$ is a closed K\"ahler manifold of complex dimension $m\geqslant2$, $\operatorname{Ric}\geqslant g$, and $q\in[1,2)\cup(2,\frac{2m}{m-1}]$. Then the Beckner--Sobolev inequality \eqref{BS} holds with $A=1$.
\end{conjecture}

The metric geometry of K\"ahler manifolds satisfying $\operatorname{Ric}\geqslant g$ is rich but still poorly understood. A natural model space is $(\mathbb{CP}^m,g_{\mathrm{FS}})$, equipped with the Fubini--Study metric normalized by $\operatorname{Ric}=g_{\mathrm{FS}}$. Indeed, the diameter of $(\mathbb{CP}^m,g_{\mathrm{FS}})$ is $\pi\sqrt{\frac{m+1}2}$, whereas the product $(\mathbb{CP}^1)^m$, equipped with the product metric and with each factor normalized by $\operatorname{Ric}=g_{\mathrm{FS}}$, has diameter $\pi\sqrt m$. Moreover,  a theorem of Bakry--Ledoux \cite{BL96} states that the $q>2$ case of \eqref{BS} implies $\operatorname{diam}(M,g)\leqslant\pi\sqrt{\frac{qA}{q-2}}$. Taking $A=1$ and letting $q$ approach the critical exponent $\frac{2m}{m-1}$, one is led to the following diameter conjecture.

\begin{conjecture}\label{conj:diameter}
    Suppose $(M,g,J)$ is a closed K\"ahler manifold of complex dimension $m\geqslant2$, and $\operatorname{Ric}\geqslant g$. Then $\operatorname{diam}(M,g)\leqslant\pi\sqrt m$, with equality if and only if $(M,g,J)$ is biholomorphically isometric to $(\mathbb{CP}^1)^m$, where $\mathbb{CP}^1$ is normalized so that $\operatorname{Ric}=g_{\mathrm{FS}}$.
\end{conjecture}

Applying the Riemannian results recalled above in real dimension $2m$ to the K\"ahler setting gives the following bounds. Bidaut-V\'eron and V\'eron \cite{BV91} obtained \eqref{BS} with $A=\frac{2m-1}{m}$. Licois and V\'eron \cite{LV95}, Fontenas \cite{Fon97}, Dolbeault, Esteban, and Loss \cite{DEL14} obtained \eqref{BS} with
$$A=A_1:=\frac{q+4m^2+4m-1}{m\{(2m-1)(q-1)+4[2m-(m-1)q]\Lambda\}}.$$
Baudoin and Munteanu \cite{BM21} exploited the special K\"ahler structure to improve the constant as follows.

\begin{theorem}[\cite{BM21}]\label{thm:BM}
    Suppose $(M,g,J)$ is a closed K\"ahler manifold of complex dimension $m\geqslant2$, and $\operatorname{Ric}\geqslant g$. Then the Beckner--Sobolev inequality \eqref{BS} holds with
    $$A=\begin{cases}
        A_2':=\frac{mq}{m+q-1},&1\leqslant q<2,\\
        A_2:=\frac{2m+q+1-2\sqrt{(m+1)[2m-(m-1)q]}}{m(q-1)},& 2<q\leqslant\frac{2m}{m-1}.
    \end{cases}$$
\end{theorem}

They also obtained a sharper upper bound on the diameter, namely
$$\operatorname{diam}(M,g)\leqslant\pi\sqrt{2m-1}(1-\frac{1}{24m}),$$
which improves the Riemannian Myers bound $\operatorname{diam}(M,g)\leqslant\pi\sqrt{2m-1}$ obtained by viewing $M$ as a real $2m$-dimensional manifold satisfying $\operatorname{Ric}\geqslant g$.

However, when $\Lambda$ is sufficiently large, $A_1<A_2$, and hence the spectral estimate yields a stronger inequality. By incorporating the first eigenvalue $\Lambda$, we establish Liouville theorems for equations \eqref{eq:Sobolev} and \eqref{eq:Beckner}.

For $0<\alpha\leqslant\frac{m+1}{m-1}$ and $\alpha\neq1$, define
\begin{equation}\label{Lambda0}
    \ell=\sqrt[4]{1-\frac{m-1}{m+1}\alpha},\quad \Lambda_0:=\frac 1 4\big(\ell+\frac 1 \ell\big)^2,
\end{equation}
\begin{equation}\label{lambda}
    0<\lambda<\begin{cases}
        \frac{m\alpha}{|\alpha-1|\{2m+\alpha+2-2\sqrt{(m+1)[m+1-(m-1)\alpha]}\}},&1\leqslant\Lambda<\Lambda_0,\\[8pt]
        \frac{m[m\alpha+4(m+1)\Lambda-2(m-1)\alpha(\Lambda+\sqrt{\Lambda(\Lambda-1)})]}{|\alpha-1|(\alpha+4m^2+4m)},& \Lambda\geqslant\Lambda_0.
    \end{cases}
\end{equation}
If $\alpha=\frac{m+1}{m-1}$, we set $\Lambda_0=+\infty$.

\begin{theorem}\label{thm:Sobolev}
    Let $(M,g,J)$ be a closed K\"ahler manifold of complex dimension $m\geqslant2$ satisfying $\operatorname{Ric}\geqslant g$, and let $\Lambda$ denote the first positive eigenvalue of $-\square$. Assume that $1<\alpha\leqslant\frac{m+1}{m-1}$, define $\Lambda_0$ by \eqref{Lambda0}, and assume that $\lambda$ satisfies \eqref{lambda}. Then every positive solution of \eqref{eq:Sobolev} is constant; more precisely, $u\equiv\lambda^{\frac{1}{\alpha-1}}$.
\end{theorem}

\begin{theorem}\label{thm:Beckner}
    Let $(M,g,J)$ be a closed K\"ahler manifold of complex dimension $m\geqslant2$ satisfying $\operatorname{Ric}\geqslant g$, and let $\Lambda$ denote the first positive eigenvalue of $-\square$. Assume that $0<\alpha<1$, define $\Lambda_0$ by \eqref{Lambda0}, and assume that $\lambda$ satisfies \eqref{lambda}. Then every positive solution of \eqref{eq:Beckner} is constant; more precisely, $u\equiv\lambda^{\frac{1}{\alpha-1}}$.
\end{theorem}

As a consequence, we obtain improved Sobolev and Beckner inequalities with explicit constants depending on the first eigenvalue.

\begin{corollary}\label{cor:Sobolev}
    Let $(M,g,J)$ be a closed K\"ahler manifold of complex dimension $m\geqslant2$ satisfying $\operatorname{Ric}\geqslant g$, and let $\Lambda$ be the first positive eigenvalue of $-\square$. Let $q\in[1,2)\cup(2,\frac{2m}{m-1}]$, and set
    $$\ell:=\sqrt[4]{\frac{2m-(m-1)q}{m+1}},\quad\Lambda_0:=\frac 1 4\big(\ell+\frac 1 \ell\big)^2.$$
    Then the Beckner--Sobolev inequality \eqref{BS} holds with
    $$A=\begin{cases}
        A_2=\frac{2m+q+1-2\sqrt{(m+1)[2m-(m-1)q]}}{m(q-1)}, & 1\leqslant\Lambda<\Lambda_0,\\
        A_3:=\frac{q+4m^2+4m-1}{m[m(q-1)+4(m+1)\Lambda-2(m-1)(q-1)(\Lambda+\sqrt{\Lambda(\Lambda-1)})]},& \Lambda\geqslant\Lambda_0.
    \end{cases}$$
    If $q=\frac{2m}{m-1}$, we set $\Lambda_0=+\infty$. For $q=1$, the value of $A$ is understood as the limit as $q\to1^+$.
\end{corollary}

For $q<\frac{2m}{m-1}$, we have $A_3=A_2$ at $\Lambda=\Lambda_0$, and $A_3$ is strictly decreasing in $\Lambda$ for $\Lambda>\Lambda_0$. Thus, Corollary \ref{cor:Sobolev} recovers the Baudoin--Munteanu bound for  $1\leqslant\Lambda\leqslant\Lambda_0$ and strictly improves it when $\Lambda>\Lambda_0$. At the critical exponent $q=\frac{2m}{m-1}$, the estimate reduces to the $A_2$ branch. Moreover, a direct calculation using $\Lambda\geqslant1$ shows that $A\leqslant A_1$ throughout the admissible parameter range. Thus, Corollary \ref{cor:Sobolev} also improves the Licois--V\'eron estimate.

By letting $q\to2$ in Corollary \ref{cor:Sobolev}, we obtain the log-Sobolev inequality.

\begin{corollary}\label{cor:log-Sobolev}
    Let $(M,g,J)$ be a closed K\"ahler manifold of complex dimension $m\geqslant2$ satisfying $\operatorname{Ric}\geqslant g$, and let $\Lambda$ be the first positive eigenvalue of $-\square$. Set
    $$\Lambda_1:=\lim_{q\to2}\Lambda_0(q)=\frac 1 4\big(\sqrt[4]{\frac{2}{m+1}}+\sqrt[4]{\frac{m+1}{2}}\big)^2.$$
    Then the log-Sobolev inequality $\fint_M f^2\log\frac{f^2}{\fint_M f^2}\leqslant 2A\fint_M|\partial f|^2$ holds with
    $$A=\begin{cases}
        \frac{2m+3-2\sqrt{2(m+1)}}{m}, & 1\leqslant\Lambda<\Lambda_1,\\
        \frac{(2m+1)^2}{m[m+2(m+3)\Lambda-2(m-1)\sqrt{\Lambda(\Lambda-1)}]},& \Lambda\geqslant\Lambda_1.
    \end{cases}$$
\end{corollary}

This paper is organized as follows. In Section \ref{sec:Riemann}, we use differential identities to streamline the proof of the Riemannian Liouville theorem, Theorem \ref{thm:Riemann-Sobolev}. In Section \ref{sec:Kahler}, we prove Theorem \ref{thm:Sobolev} using arguments analogous to those developed in Section \ref{sec:Riemann}. The proofs of Theorems \ref{thm:Riemann-Beckner} and \ref{thm:Beckner} require only minor modifications, which are explained at the beginning of the corresponding sections. Finally, in Section \ref{sec:diameter}, we discuss upper bounds for the diameter of $(M,g)$.
\section{The Riemannian Case: Proof of Theorem \ref{thm:Riemann-Sobolev}}\label{sec:Riemann}

In this section, we present a new proof of Theorem \ref{thm:Riemann-Sobolev} which illustrates our method. Differentiating the equations $-\Delta u+\lambda u=u^\alpha$ and $\Delta u+\lambda u=u^\alpha$ yields, respectively,
$$\nabla_i\Delta u=\alpha\frac{\Delta u}{u}u_i-(\alpha-1)\lambda u_i\text{ and }\nabla_i\Delta u=\alpha\frac{\Delta u}{u}u_i+(\alpha-1)\lambda u_i.$$
These differentiated identities are the only place where the equations enter the argument. Therefore, the proof of Theorem \ref{thm:Riemann-Beckner} follows from the same argument, with $\lambda$ replaced by $-\lambda$, and is omitted.

\begin{proof}[Proof of Theorem \ref{thm:Riemann-Sobolev}]
    Define two trace-free tensors
    $$E_{ij}:=u_{ij}-\frac{\Delta u}{n}g_{ij},\quad L_{ij}:=\frac{u_iu_j}{u}-\frac{|\nabla u|^2}{nu}g_{ij}.$$
    Let $\beta\in\mathbb R$ be chosen later. Direct computations yield the following identities:
    $$u^{-\beta}\nabla^i(u^\beta E_{ij}u^j)=|E|^2+\beta \frac{E_{ij}u^iu^j}{u}+\frac{n-1}{n}\alpha\frac{\Delta u|\nabla u|^2}{u}+R_{ij}u^iu^j-\frac{n-1}{n}(\alpha-1)\lambda|\nabla u|^2,$$
    $$u^{-\beta}\nabla^i(u^{\beta-1}|\nabla u|^2u_i)=2\frac{E_{ij}u^iu^j}{u}+\frac{n+2}{n}\frac{\Delta u|\nabla u|^2}{u}+(\beta-1)\frac{|\nabla u|^4}{u^2},$$
    $$u^{-\beta}\nabla^i(u^\beta\Delta uu_i)=(\Delta u)^2+(\alpha+\beta)\frac{\Delta u|\nabla u|^2}{u}-(\alpha-1)\lambda|\nabla u|^2.$$
    Here we have used $\nabla_i\Delta u=\alpha\frac{\Delta u}{u}u_i-(\alpha-1)\lambda u_i$.

    Taking a linear combination of these three identities with coefficients $1,c,d$, where $c,d\in\mathbb R$ are to be determined, yields

    \begin{align}
        \begin{split}\label{id:Riemann}
            &u^{-\beta}\nabla^i\big(u^\beta(E_{ij}u^j+c\frac{|\nabla u|^2}{u}u_i+d\Delta u u_i)\big)\\
            ={}&|E|^2+(\beta+2c)\frac{E_{ij}u^iu^j}{u}+R_{ij}u^iu^j-(\frac{n-1}{n}+d)(\alpha-1)\lambda|\nabla u|^2\\
            &+d(\Delta u)^2+\big(\frac{n-1}{n}\alpha+\frac{n+2}{n}c+d(\alpha+\beta)\big)\frac{\Delta u|\nabla u|^2}{u}+c(\beta-1)\frac{|\nabla u|^4}{u^2}\\
            ={}&\big|E+(\frac\beta 2+c)L\big|^2+\mathscr R+d(\Delta u+\frac\beta 2\frac{|\nabla u|^2}{u})^2+\big(n-1-(\frac{n-1}{n}+d)(\alpha-1)\lambda\big)|\nabla u|^2\\
            &+\big((\frac{n-1}{n}+d)\alpha+\frac{n+2}{n}c\big)\frac{\Delta u|\nabla u|^2}{u}+\big(c(\beta-1)-\frac{n-1}{n}(\frac\beta 2+c)^2-\frac{d}{4}\beta^2\big)\frac{|\nabla u|^4}{u^2},
        \end{split}
    \end{align}
    where $\mathscr R=R_{ij}u^iu^j-(n-1)|\nabla u|^2\geqslant0$ by the assumption $\operatorname{Ric}\geqslant(n-1)g$. To eliminate the coefficients of $\frac{\Delta u|\nabla u|^2}{u}$ and $\frac{|\nabla u|^4}{u^2}$, we solve the resulting two algebraic equations and obtain
    $$c=\frac{(n+2)\beta^2+4\alpha\beta-4n\alpha}{4(n-1)\alpha},\quad d=-\frac{(n+2)^2\beta^2+4(n+2)\alpha\beta+4\alpha[(n-1)^2\alpha-n(n+2)]}{4n(n-1)\alpha^2}.$$
    The choice $\beta=-\frac{2\alpha}{n+2}$ maximizes $d$, giving $d=\frac{n+2-(n-2)\alpha}{(n-1)\alpha}\geqslant0$. Substituting these choices of $\beta$, $c$, and $d$ into \eqref{id:Riemann}, we obtain
    \begin{align}
        \begin{split}\label{ineq:Riemann}
            &u^{-\beta}\nabla^i\big(u^\beta(E_{ij}u^j+c\frac{|\nabla u|^2}{u}u_i+d\Delta u u_i)\big)\\
            \geqslant{}&d(\Delta u+\frac\beta 2\frac{|\nabla u|^2}{u})^2+\big(n-1-\frac{\alpha+n^2+2n}{(n-1)n\alpha}(\alpha-1)\lambda\big)|\nabla u|^2,
        \end{split}
    \end{align}
    
    Applying the spectral inequality $\int_M|\Delta f|^2\geqslant\Lambda\int_M|\nabla f|^2$ to $f=u^{\frac\beta 2+1}$, we obtain
    $$\int_M u^\beta\big(\Delta u+\frac\beta 2\frac{|\nabla u|^2}{u}\big)^2\geqslant\Lambda\int_M u^\beta|\nabla u|^2.$$
    Multiplying \eqref{ineq:Riemann} by $u^\beta$ and integrating over $M$, the divergence term vanishes. We then apply the preceding spectral inequality to the square term. Recalling the definition of $\theta$, we obtain
    $$0\geqslant\frac{\alpha+n^2+2n}{(n-1)n\alpha}\big(\theta n+(1-\theta)\Lambda-(\alpha-1)\lambda\big)\int_Mu^\beta|\nabla u|^2.$$
    By the assumed upper bound on $\lambda$, the coefficient on the right-hand side is strictly positive. Consequently, $\int_Mu^\beta|\nabla u|^2=0$. Since $u>0$, it follows that $\nabla u\equiv0$, and hence $u$ is constant. Substituting this constant into the equation gives $u\equiv\lambda^{\frac{1}{\alpha-1}}$.
\end{proof}

The essential point of the proof is that, after a suitable choice of the parameters $c$, $d$, and $\beta$, the terms involving $(\Delta u)^2$, $\frac{\Delta u|\nabla u|^2}{u}$, and $\frac{|\nabla u|^4}{u^2}$ can be organized into the square $(\Delta u+\frac\beta 2\frac{|\nabla u|^2}{u})^2$. Applying the spectral inequality to $u^{1+\frac\beta 2}$ then converts this square term into a lower bound involving the first eigenvalue $\Lambda$, thereby enlarging the admissible range of $\lambda$. This observation motivates the analogous construction in the K\"ahler setting developed in the next section.
\section{The K\"{a}hler Case: Proof of Theorem \ref{thm:Sobolev}}\label{sec:Kahler}

As in Section \ref{sec:Riemann}, we prove only Theorem \ref{thm:Sobolev}. In the sublinear case, differentiation of the equation yields the same identities as those below, with $\lambda$ replaced by $-\lambda$. We therefore omit the proof of Theorem \ref{thm:Beckner}.

\begin{proof}[Proof of Theorem \ref{thm:Sobolev}]
    Differentiating the equation \eqref{eq:Sobolev} yields
    $$\nabla_i\square u=\alpha\frac{\square u}{u}u_i-(\alpha-1)\lambda u_i.$$

    Define two trace-free tensors
    $$E_{i\overline j}:=u_{i\overline j}-\frac{\square u}{m}g_{i\overline j},\quad L_{i\overline j}:=\frac{u_iu_{\overline j}}{u}-\frac{|\partial u|^2}{mu}g_{i\overline j}.$$
    Let $\beta$ be a real constant to be chosen later. Direct computations yield the following identities:
    $$u^{-\beta}\nabla^i(u^\beta u_{ij}u^j)=|\partial\partial u|^2+\beta\frac{u_{ij}u^iu^j}{u}+\alpha\frac{\square u|\partial u|^2}{u}+R_{i\overline j}u^iu^{\overline j}-(\alpha-1)\lambda|\partial u|^2,$$
    $$u^{-\beta}\nabla^i(u^\beta E_{i\overline j}u^{\overline j})=|E|^2+\beta\frac{E_{i\overline j}u^iu^{\overline j}}{u}+\frac{m-1}{m}\alpha\frac{\square u|\partial u|^2}{u}-\frac{m-1}{m}(\alpha-1)\lambda|\partial u|^2,$$
    $$u^{-\beta}\nabla^i(u^{\beta-1}|\partial u|^2u_i)=\frac{u_{\overline i\overline j}u^{\overline i}u^{\overline j}}{u}+\frac{E_{i\overline j}u^iu^{\overline j}}{u}+\frac{m+1}{m}\frac{\square u|\partial u|^2}{u}+(\beta-1)\frac{|\partial u|^4}{u^2},$$
    $$u^{-\beta}\nabla^i(u^\beta\square uu_i)=(\square u)^2+(\alpha+\beta)\frac{\square u|\partial u|^2}{u}-(\alpha-1)\lambda|\partial u|^2.$$
    Since $E_{i\bar j}$ and $R_{i\bar j}$ are Hermitian, both $E_{i\overline j}u^iu^{\overline j}$ and $R_{i\overline j}u^iu^{\overline j}$ are real-valued. Taking a linear combination of the real parts of these identities with coefficients $1,b,c,d$, where $b,c,d\in\mathbb R$ are to be determined, we obtain
    \begin{align}
        \begin{split}\label{id:Kahler}
            &u^{-\beta}\operatorname{Re}\nabla^i\big(u^\beta(u_{ij}u^j+bE_{i\overline j}u^{\overline j}+c\frac{|\partial u|^2}{u}u_i+d\square u u_i)\big)\\
            ={}&|\partial\partial u|^2+b|E|^2+(\beta+c)\operatorname{Re}\frac{u_{ij}u^iu^j}{u}+(b\beta+c)\frac{E_{i\overline j}u^iu^{\overline j}}{u}\\
            &+R_{i\overline j}u^iu^{\overline j}-(1+\frac{m-1}{m}b+d)(\alpha-1)\lambda|\partial u|^2+d(\square u)^2\\
            &+\big((1+\frac{m-1}{m}b+d)\alpha+d\beta+\frac{m+1}{m}c\big)\frac{\square u|\partial u|^2}{u}+c(\beta-1)\frac{|\partial u|^4}{u^2}\\
            ={}&\big|\partial\partial u+\frac{\beta+c}{2}\frac{\partial u\otimes\partial u}{u}\big|^2+b\big|E+\frac{b\beta+c}{2b}L\big|^2+\mathscr R+d(\square u+\frac\beta 2\frac{|\partial u|^2}{u})^2\\
            &+\big(1-(1+\frac{m-1}{m}b+d)(\alpha-1)\lambda\big)|\partial u|^2+\big((1+\frac{m-1}{m}b+d)\alpha+\frac{m+1}{m}c\big)\frac{\square u|\partial u|^2}{u}\\
            &+\big(c(\beta-1)-\frac{(\beta+c)^2}{4}-\frac{m-1}{4mb}(b\beta+c)^2-\frac{d\beta^2}{4}\big)\frac{|\partial u|^4}{u^2},
        \end{split}
    \end{align}
    where $\mathscr R=R_{i\overline j}u^iu^{\overline j}-|\partial u|^2\geqslant0$ by the assumption $\operatorname{Ric}\geqslant g$. We shall eventually restrict $b$ to a positive interval. For fixed $b>0$, we require the coefficients of $\frac{\square u|\partial u|^2}{u}$ and $\frac{|\partial u|^4}{u^2}$ to vanish and then maximize $d$ with respect to $\beta$. This gives $\beta=-\frac{\alpha}{m+1}$, $c=-\frac{(\alpha+4m^2+4m)b}{(m+1)(mb+m-1)}$, and $d=\frac{4(m+1)b-(m-1)\alpha(b+1)^2}{(mb+m-1)\alpha}$.

    The condition $d\geqslant0$ is equivalent to   $b_1\leqslant b\leqslant b_2$, where
    $$b_1:=\frac{2(m+1)}{\alpha(m-1)}-1-\frac 2\alpha\sqrt{\frac{m+1}{m-1}\big(\frac{m+1}{m-1}-\alpha\big)},$$
    $$b_2:=\frac{2(m+1)}{\alpha(m-1)}-1+\frac 2\alpha\sqrt{\frac{m+1}{m-1}\big(\frac{m+1}{m-1}-\alpha\big)}=\frac{1}{b_1}.$$
    Since $\alpha\leqslant\frac{m+1}{m-1}$, we have $b_2\geqslant\frac{2(m+1)}{\alpha(m-1)}-1\geqslant1$ and $b_1=\frac{1}{b_2}>0$.
    
    For $b\in[b_1,b_2]$, we have $b>0$ and $d\geqslant0$. Consequently, all the square terms and $\mathscr R$ in \eqref{id:Kahler} are nonnegative. Dropping the first two square terms and $\mathscr R$, while retaining the square term with coefficient $d$, we obtain
    \begin{align}
        \begin{split}\label{ineq:Kahler}
            &u^{-\beta}\operatorname{Re}\nabla^i\big(u^\beta(u_{ij}u^j+bE_{i\overline j}u^{\overline j}+c\frac{|\partial u|^2}{u}u_i+d\square u u_i)\big)\\
            \geqslant{}&d(\square u+\frac\beta 2\frac{|\partial u|^2}{u})^2+\big(1-(1+\frac{m-1}{m}b+d)(\alpha-1)\lambda\big)|\partial u|^2.
        \end{split}
    \end{align}
    Applying the spectral inequality $\int_M|\square f|^2\geqslant\Lambda\int_M|\partial f|^2$ to $f=u^{\frac\beta 2+1}$, we obtain
    $$\int_M u^\beta\big(\square u+\frac\beta 2\frac{|\partial u|^2}{u}\big)^2\geqslant\Lambda\int_M u^\beta|\partial u|^2.$$
    Multiplying \eqref{ineq:Kahler} by $u^\beta$ and integrating over $M$, the divergence term vanishes. Using the spectral inequality to estimate the first term on the right-hand side, we obtain $0\geqslant$
    $$\frac{m(m-1)\alpha[1-\Lambda(b+1)^2]+4m(m+1)b[\Lambda-(\alpha-1)\lambda]+\alpha b[m^2-(\alpha-1)\lambda]}{m\alpha(mb+m-1)}\int_Mu^\beta|\partial u|^2.$$

    To conclude that $u$ is constant, it is therefore sufficient to verify that
    $$\lambda<f(b):=\frac{m\{(m-1)\alpha[1-(b+1)^2\Lambda]+[m\alpha+4(m+1)\Lambda]b\}}{(\alpha-1)(\alpha+4m^2+4m)b},\quad b\in[b_1,b_2].$$

    If $\alpha=\frac{m+1}{m-1}$, then $b_1=b_2=1$ and $\Lambda_0=+\infty$. Taking $b=1$ directly gives the first bound in \eqref{lambda}. Hence, in the remainder of the proof, we assume that $1<\alpha<\frac{m+1}{m-1}$, so that $0<b_1<1<b_2$.
    
    To maximize $f(b)$, we compute $f'(b)=\frac{m(m-1)\alpha(\Lambda-1-\Lambda b^2)}{(\alpha-1)(\alpha+4m^2+4m)b^2}$. The only positive critical point of $f$ is $b_0:=\sqrt{1-\frac 1 \Lambda}$ when $\Lambda>1$, and we have $b_0<1<b_2$. Note that
    $$f(b_1)=\frac{m\alpha}{(\alpha-1)\{2m+\alpha+2-2\sqrt{(m+1)[m+1-(m-1)\alpha]}\}},$$
    $$f(b_0)=\frac{m[m\alpha+4(m+1)\Lambda-2(m-1)\alpha(\Lambda+\sqrt{\Lambda(\Lambda-1)})]}{(\alpha-1)(\alpha+4m^2+4m)}.$$
    
    \textbf{Case 1. $1\leqslant\Lambda<\Lambda_0$.} By the definition of $\Lambda_0$, the inequality $\Lambda<\Lambda_0$ is equivalent to $b_0<b_1$ when $\Lambda>1$. Thus, $f$ is decreasing on $(b_1,b_2)$. The proof is completed by taking $b=b_1$.
    
    \textbf{Case 2. $\Lambda\geqslant\Lambda_0$.} We have $b_1\leqslant b_0$. Thus, $f$ is increasing on $(b_1,b_0)$ and decreasing on $(b_0,b_2)$. The proof is completed by taking $b=b_0$.

    Thus, under the assumed bound on $\lambda$, we can choose $b\in[b_1,b_2]$ such that $\lambda<f(b)$. It follows from the preceding integral inequality that $\int_Mu^\beta|\partial u|^2=0$. Since $u>0$, we obtain $\partial u\equiv0$, and hence $u$ is constant. Substituting into \eqref{eq:Sobolev}, we conclude that $u\equiv\lambda^{\frac{1}{\alpha-1}}$.
\end{proof}
\section{Upper Bounds for the Diameter}\label{sec:diameter}

In this section, we discuss several upper bounds for the diameter. Recall that Baudoin and Munteanu \cite{BM21} obtained the diameter estimate
$$\operatorname{diam}(M,g)\leqslant\pi\sqrt{2m-1}(1-\frac{1}{24m}).$$
Their argument involves several estimates that are not optimized with respect to the exponent $q$. By optimizing the resulting bound directly, we obtain a sharper asymptotic estimate.

Combining Theorem \ref{thm:BM} with the Bakry--Ledoux diameter estimate \cite{BL96}, we obtain
$$\operatorname{diam}(M,g)\leqslant\pi\sqrt{f(q)},$$
where $f(q):=\frac{qA_2(q)}{q-2}$, $A_2(q):=\frac{2m+q+1-2\sqrt{(m+1)[2m-(m-1)q]}}{m(q-1)}$.

To minimize $f(q)$ over $q\in(2,\frac{2m}{m-1}]$, we compute
\begin{align*}
    f'(q)={}&\frac{-(m-1)q^3+(m+3)q^2+6(m-1)q-8m}{m(q-1)^2(q-2)^2}\sqrt{\frac{m+1}{2m-(m-1)q}}\\
    &+\frac{2[1+2m+2q-(m+2)q^2]}{m(q-1)^2(q-2)^2}.
\end{align*}
After isolating the square-root term and squaring, the equation $f'(q)=0$ reduces to a cubic equation. Checking the original unsquared equation rules out any extraneous roots introduced by squaring. It follows that $f$ is strictly decreasing on $(2,q_0)$ and strictly increasing on $(q_0,\frac{2m}{m-1})$, where $q_0$ is the unique root in the
admissible interval of
$$(m-1)^2q_0^3+2(m-1)(m+4)q_0^2-4(2m-1)(2m+3)q_0+4(4m^2+4m-1)=0.$$
In fact, the explicit form of $q_0$ is
$$q_0=\frac{4\sqrt{(m+1)(13m+7)}}{3(m-1)}\cos\Big(\frac\pi 3-\frac 1 3\arccos\frac{184m^2+224m+67}{4\sqrt{m+1}(13m+7)^{\frac 3 2}}\Big)-\frac{2(m+4)}{3(m-1)}.$$
Since $q_0=2+\frac 2 m+\frac{1}{m^2}+O(\frac{1}{m^3})$ as $m\to\infty$, we obtain the sharper asymptotic estimate
$$\operatorname{diam}(M,g)\leqslant\pi\sqrt{f(q_0)}=\pi\sqrt{2m-1}\big(1-\frac{1}{4m}+O(\frac{1}{m^2})\big)\quad(m\to\infty).$$

Nevertheless, a different approach, due to Jiaping Wang, yields $\operatorname{diam}(M,g)\leqslant\pi\sqrt{2m-1}D$ for any $D\in(0,1)$ satisfying $\frac{\int_0^{\frac\pi 2}\sin^2 t\sin^{2m-1}(Dt)\,\mathrm dt}{\int_0^{\frac\pi 2}\cos^2 t\sin^{2m-1}(Dt)\,\mathrm dt}\geqslant2(2m-1)D^2$; see \cite{BM21}. It is a better bound for sufficiently large $m$. Baudoin and Munteanu \cite{BM21} used this characterization to derive
$$\operatorname{diam}(M,g)\leqslant\pi\sqrt{2m-1}\big(1-\frac{1}{200\sqrt m\log m}\big).$$
This estimate is asymptotically stronger than the one obtained by optimizing $A_2(q)$.

Finally, Corollary \ref{cor:Sobolev} yields a further improvement when the first eigenvalue is large. For every fixed $q\in(2,\frac{2m}{m-1})$, we have $\lim_{\Lambda\to+\infty}A_3(q,\Lambda)=0$. Combining this with the Bakry--Ledoux estimate gives $\operatorname{diam}(M,g)\leqslant\pi\sqrt{\frac{qA_3(q,\Lambda)}{q-2}}\to0$ as $\Lambda\to+\infty$. Thus, for sufficiently large $\Lambda$, the spectral estimate improves the diameter bounds obtained from $A_2$ and from Jiaping Wang.

\bigskip 

We end this paper with two open problems for future studies. We have seen that the geometric approach yields $C=\frac{1}{n\kappa\tau^2}$, which is better than $\frac{1}{n\kappa}$ when $(M,g) $ is not the model space $\mathbb S^n(\frac{1}{\sqrt\kappa})$. On the other hand, the analytic approach based on the study of positive solutions to $-\Delta u+\lambda u=u^q$ yields improvements only for $q<\frac{2n}{n-2}$. In the critical case $q=\frac{2n}{n-2}$, the constants in \cite{BV91} and Theorem \ref{thm:Riemann-Sobolev} degenerate to $\frac{1}{n\kappa}$. Is this the limitation of the methods or is there a reason for this? Besides, in the Kahler case, the current estimates are far from optimal. What other structures can one exploit to get better estimates?

\bigskip
\noindent\textbf{Acknowledgments.} The authors would like to thank Professor Xi-Nan Ma for the constant encouragement in this paper. Tian Wu is supported by the National Key Research and Development Project (Grant No. 2025YFA1017600), the Fundamental Research Funds for the Central Universities (Grant No. WK0010250106), and the Open Research Fund of Hubei Key Laboratory of Mathematical Sciences (Central China Normal University, Wuhan 430079, P. R. China).

\noindent \textbf{Research ethics:} Not applicable.

\noindent \textbf{Informed consent:} Not applicable.

\noindent \textbf{Author contributions:} All authors have accepted responsibility for the entire content of this\\ manuscript and approved its submission.

\noindent \textbf{Use of Large Language Models, AI and Machine Learning Tools:} None declared.

\noindent \textbf{Conflict of interest:} The authors state no conflict of interest.

\noindent \textbf{Data availability:} Not applicable.


\footnotesize{
    Contact information:
    \begin{itemize}
        \item Sayantan Chakraborty, Department of Mathematics, Michigan State University, East Lansing, MI 48824, USA. Email: \emph{chakra85@msu.edu}
        \item Xiaodong Wang, Department of Mathematics, Michigan State University, East Lansing, MI 48824, USA. Email: \emph{xwang@msu.edu}
        \item Tian Wu, School of Mathematical Sciences, University of Science and Technology of China, Hefei, Anhui, 230026, People's Republic of China. Email: \emph{wt1997@ustc.edu.cn}
    \end{itemize}
}

\end{document}